\documentclass[11pt,oneside,english]{amsart}
\usepackage[T1]{fontenc}
\usepackage[latin9]{inputenc}
\usepackage{geometry}
\usepackage{textcomp}
\usepackage{amstext}
\usepackage{amsthm}
\usepackage{amssymb}

\makeatletter
\numberwithin{equation}{section}
\numberwithin{figure}{section}
\theoremstyle{plain}
\newtheorem{thm}{\protect\theoremname}[section]
\theoremstyle{definition}
\newtheorem{example}[thm]{\protect\examplename}
\theoremstyle{remark}
\newtheorem{rem}[thm]{\protect\remarkname}
\theoremstyle{definition}
\newtheorem{defn}[thm]{\protect\definitionname}
\theoremstyle{plain}
\newtheorem{lem}[thm]{\protect\lemmaname}
\theoremstyle{plain}
\newtheorem{prop}[thm]{\protect\propositionname}

\def\makebbb#1{
    \expandafter\gdef\csname#1\endcsname{
        \ensuremath{\Bbb{#1}}}
}\makebbb{R}\makebbb{N}\makebbb{Z}\makebbb{C}\makebbb{H}\makebbb{E}\makebbb{H}\makebbb{P}\makebbb{B}\makebbb{Q}\makebbb{E}\makebbb{E}
\makebbb{H}
\usepackage{babel}

\usepackage{babel}

\makeatother

\usepackage{babel}

\makeatother

\usepackage{babel}
\providecommand{\definitionname}{Definition}
\providecommand{\examplename}{Example}
\providecommand{\lemmaname}{Lemma}
\providecommand{\propositionname}{Proposition}
\providecommand{\remarkname}{Remark}
\providecommand{\theoremname}{Theorem}

\begin{document}
\title{Measure preserving holomorphic vector fields, invariant anti-canonical
divisors and Gibbs stability }
\author{Robert J.Berman}
\dedicatory{Dedicated to Laszlo Lempert on the occasion of his 70th anniversary}
\begin{abstract}
Let $X$ be a compact complex manifold whose anti-canonical line bundle
$-K_{X}$ is big. We show that $X$ admits no non-trivial holomorphic
vector fields if it is Gibbs stable (at any level). The proof is based
on a vanishing result for measure preserving holomorphic vector fields
on $X$ of independent interest. As an application it shown that,
in general, if $-K_{X}$ is big, there are no holomorphic vector fields
on $X$ that are tangent to a non-singular irreducible anti-canonical
divisor $S$ on $X.$ More generally, the result holds for varieties
with log terminal singularities and log pairs. Relations to a result
of Berndtsson about generalized Hamiltonians and coercivity of the
quantized Ding functional are also pointed out.
\end{abstract}

\maketitle

\section{Introduction}

Let $X$ be a compact complex manifold and denote by $K_{X}$ its
canonical line bundle, i.e. the top exterior power of the holomorphic
cotangent bundle. We will use additive notation for tensor products.
For a given positive integer $k$ set 
\begin{equation}
N_{k}:=\dim_{\C}H^{0}(X,-kK_{X})\label{eq:def of N k}
\end{equation}
 and assume that $N_{k}>0.$ The $N_{k}-$fold products $X^{N_{k}}$
of $X$ come with a natural effective anti-canonical divisor $\Q-$divisor
$\mathcal{D}_{N_{k}},$ whose support is defined by 
\[
\left\{ (x_{1},...x_{N_{k}}):\,\exists s_{k}\in H(X,-kK_{X}):\,\,s_{k}(x_{i})=0,\,\,\,\forall i,\,\,\,s_{k}\not\equiv0\right\} 
\]
In classical terminology the support of $\mathcal{D}_{N_{k}}$ thus
consists of all configurations $(x_{1},...x_{N_{k}})$ of points on
$X^{N_{k}}$ which are in ``bad position'' with respect to $H^{0}(X,-kK_{X}).$
Equivalently, the divisor $\mathcal{D}_{k}$ in $X^{N_{k}}$ may be
defined as $k^{-1}$ times the zero-locus (including multiplicities)
of the holomorphic section $\det S^{(k)}$ of $-kK_{X^{N_{k}}}\rightarrow X^{N_{k}}$
defined as the following Slater determinant
\begin{equation}
(\det S^{(k)})(x_{1},x_{2},...,x_{N_{k}}):=\det\left(s_{i}^{(k)}(x_{j})\right),\label{eq:slater determinant}
\end{equation}
 in terms of a given basis $s_{1}^{(k)},...,s_{N_{k}}^{(k)}$ in $H^{0}(X,-kK_{X})$
(see Lemma \ref{lem:The-zero-locus-of det section}).

Following \cite{berm8 comma 5}, we will say that $X$ is \emph{Gibbs
stable at level $k$} if $N_{k}>0$ and the $\Q-$divisor $\mathcal{D}_{N_{k}}$
on $X^{N_{k}}$ is klt (Kawamata Log Terminal). This means, loosely
speaking, that the divisor $\mathcal{D}_{N_{k}}$ has mild singularities
in the sense of the Minimal Model Program in birational algebraic
geometry \cite{ko0}. Furthermore, $X$ is called \emph{Gibbs stable}
if it is Gibbs stable for any sufficiently large level $k.$ The definition
of Gibbs stability is motivated by the probabilistic approach to the
construction of Kähler-Einstein metrics on Fano manifolds introduced
in \cite{berm8 comma 5}, as briefly recalled in Section \ref{subsec:Relations-to-the}.
\begin{thm}
\label{thm:gibbs intro}Assume that $-K_{X}$ is big and that $X$
is Gibbs stable at some level $k.$ Then $X$ admits no non-trivial
holomorphic vector fields.
\end{thm}

The special case when $X$ is Fano, i.e. $-K_{X}$ is ample and $k$
is taken sufficiently large appears in \cite[Prop 6.5]{berm8 comma 5},
but, unfortunately, the proof was inaccurately formulated. Here the
proof will be corrected and - as a bonus - two more proofs of Theorem
\ref{thm:gibbs intro} will be provided (of statements that are a
bit weaker). More generally, we will show that Theorem \ref{thm:gibbs log }
holds in the more general setting of log pairs $(X,\Delta),$ where
$X$ is a normal variety endowed with a $\Q-$divisor $\Delta$ (see
Section \ref{sec:Extension-to-the}). In particular, taking $\Delta$
to be trivial shows that Theorem \ref{thm:gibbs intro} holds more
generally when $X$ is a normal variety. We recall that the class
of varieties for which $-K_{X}$ is big contains the vast class of
Fano type varieties (such as toric varieties). However, in general,
the anti-canonical ring of a variety $X$ with big $-K_{X}$ need
not be finitely generated, even if $X$ is non-singular \cite[Ex 3.7]{c-g}.

The starting point of the proofs is the standard analytic characterization
of the klt condition of a divisor, which in the present setup means
that the measure on $X^{N_{k}}$ corresponding to the anti-canonical
divisor $\mathcal{D}_{k}$, symbolically denoted by $\left|\det S^{(k)}\right|^{-2/k}$,
is in $L_{loc}^{p}$ for some $p>1.$ In particular, the corresponding
probability measure $\mu^{(N_{k})}$ on $X^{N_{k}}$ is then well-defined:
\begin{equation}
\mu^{(N_{k})}:=\frac{1}{\mathcal{Z}_{N_{k}}}\left|\det S^{(k)}\right|^{-2/k},\,\,\,\mathcal{Z}_{N_{k}}:=\int_{X^{N_{k}}}\left|\det S^{(k)}\right|^{-2/k}.\label{eq:def of mu N and Z N}
\end{equation}
This measure is canonically attached to $X$ and thus invariant under
the diagonal action on $X^{N_{k}}$ by any biholomorphism of $X$
(see Lemma \ref{lem:The-zero-locus-of det section}). In particular,
$\mu^{(N_{k})}$ is invariant under the diagonal action on $X^{N_{k}}$
of the flow of any given holomorphic vector field on $X.$ Hence,
theorem \ref{thm:gibbs intro} follows from the following result of
independent interest (applied to $X^{N_{k}})$:
\begin{thm}
\label{thm:mu intro}Let $X$ be a compact manifold such that $-K_{X}$
is big and $V^{1,0}$ a holomorphic vector field on $X.$ If the flow
of $V^{1,0}$ preserves a measure $\mu$ on $X$ with a density in
$L_{loc}^{p}$ for some $p>1,$ then $V^{1,0}$ vanishes identically. 
\end{thm}

In the case when the local density of $\mu$ is of the form $e^{-\phi}$
for a local function $\phi$ of the form 
\begin{equation}
\phi=\sum_{i=1}^{m}\lambda_{i}\log|f_{i}|+\psi\label{eq:phi is psh plus log intro}
\end{equation}
 with $f_{i}$ holomorphic, $\lambda_{i}\in\R$ and $\psi$ plurisubharmonic
it is enough to assume that the density of $\mu$ is in $L_{loc}^{1},$
i.e. that $\mu$ has finite total mass. Indeed, by the resolution
of Demailly's strong openness conjecture \cite{g-z} this implies
$L_{loc}^{p}-$integrability for some $p>1$ (this is a strengthening
of Demailly-Kollar's openness conjecture, concerning the case when
$\lambda_{i}\geq0,$ previously established in \cite{bern}). The
assumption that $-K_{X}$ be big in the previous proposition can not
be dispensed with, as illustrated by the case when $X=\C/(\Z+i\Z)$
and $\mu$ is the measure on $X$ induced by Lebesgue measure on $\C$
(which is invariant under holomorphic translations). This example
also shows, by taking products, that it is not enough to assume that
$K_{X}$ is ``next to big'' in the sense that it has next to maximal
Kodaira dimension. On the other hand, in the opposite case that $K_{X}$
is big it is well-known that $X$ admits no non-trivial holomorphic
vector fields at all (see Remark \ref{rem:general type}).

In the second proof of Theorem \ref{thm:gibbs intro} we observe that
Theorem \ref{thm:mu intro} can be deduced from a result of Berndtsson
\cite{bern} about generalized Hamiltonians in the case that $-K_{X}$
is big and nef and the current defined by the Ricci curvature of $\mu$
is positive (i.e. when $\phi$ is psh). This approach, however, does
not apply in the more general setting of log pairs $(X,\Delta)$ when
$X$ is singular or the divisor $\Delta$ is non-effective. Conversely,
we show that Theorem \ref{thm:mu intro} implies a partial generalization
of Berndtsson's result.

The third proof of Theorem \ref{thm:gibbs intro} exploits the connection
between Gibbs stability at level $k$ and coercivity of the quantized
Ding functional established in \cite{rtz,ber13} and applies when
$k$ is taken sufficiently large to ensure that the corresponding
Kodaira map is a bimeromorphism.

In the final section the concept of (log) Gibbs stability is applied
to prove the following
\begin{thm}
\label{thm:not klt intro}Let $X$ be a compact complex manifold such
that $-K_{X}$ is big and $s$ a non-trivial holomorphic section of
$-K_{X},$ whose zero-locus $S$ is a non-singular connected subvariety
of $X$ . Then $S$ is not preserved by the flow of any non-trivial
holomorphic vector field on $X.$ More generally, the result holds
when the varieties $X$ and $S$ have log terminal singularities and
$S,$ viewed as an anti-canonical divisor on $X,$ is irreducible.
\end{thm}

In this case the divisor $S$ is not klt on $X,$ i.e. the corresponding
measure $\left|s\right|^{-2/k}$ is not in $L_{loc}^{1},$ which prevents
a direct application of Theorem \ref{thm:mu intro}. But for any sufficiently
small positive number $\gamma$ the anti-canonical divisor
\begin{equation}
\gamma\mathcal{D}_{N_{k}}+(1-\gamma)\left(\pi_{1}^{*}S+\cdots\pi_{N_{k}}^{*}S\right),\label{eq:perturbed divisor intro}
\end{equation}
 on $X^{N_{k}}$ - where $\pi_{i}$ denotes the projection from $X^{N_{k}}$
to the $i$ th factor $X$ - is klt for any fixed $k$ such that $N_{k}>1$
(this, essentially, amounts to Gibbs stability of the log pair $(X,(1-\gamma)S$)
). Since, $\mathcal{D}_{N_{k}}$ is invariant under the diagonal flow
of any holomorphic vector field on $X$ Theorem \ref{thm:not klt intro}
thus follows from Theorem \ref{thm:mu intro} applied to $X^{N_{k}},$
endowed with the measure induced by the anti-canonical divisor in
formula \ref{eq:perturbed divisor intro}. 

In the case when $-K_{X}$ is ample (and $S$ is assumed non-singular,
which implies connectedness when $-K_{X}$ is ample and $\dim X>1$)
Theorem \ref{thm:not klt intro} was shown in \cite{ber1}, confirming
a conjecture of Donaldson \cite{do2}. The proof in \cite{ber1} uses
the coercivity of the corresponding log Mabuchi functional (which,
loosely speaking, can be viewed as the large $k$ limit of the present
proof). A different proof was then given in \cite[Thm 2.8]{s-w},
including a generalization to the case when $S$ is the non-singular
zero-locus of a holomorphic section $s_{k}$ of $-kK_{X}$ for $k>1.$
Note, however, that for $k>1$ one can directly apply Theorem \ref{thm:mu intro}
to the measure $\left|s_{k}\right|^{-2/k}$ on $X$ when $-K_{X}$
is big (without assuming connectedness of $S$). The proof in \cite{s-w}
does not seem to generalize to the case when $-K_{X}$ is only assumed
big (or either $X$ or $S$ is singular). Indeed, it relies, in particular,
on the Kodaira vanishing $H^{\dim X-1}(X,\mathcal{O}_{X})=0,$ which
does not hold when $-K_{X}$ is merely big (see Example \ref{exa:X surface}). 

Note that, by adjunction, the variety $S$ appearing in Theorem \ref{thm:not klt intro}
is Calabi-Yau, i.e. $K_{S}$ is trivial. This construction of Calabi-Yau
varieties $S$ and the corresponding log Calabi-Yau variety $X-S$
play a prominent role in mirror symmetry (see \cite{ba} for the case
when $X$ is a Gorenstein toric Fano variety and \cite{au} where
$-K_{X}$ is only assumed to be effective).

\subsection{\label{subsec:Relations-to-the}Relations to the Yau-Tian-Donaldson
conjecture and K-stability}

We conclude the introduction with a discussion of some relations to
the Yau-Tian-Donaldson conjecture, that motivate some of the considerations
in the present work. According to the Yau-Tian-Donaldson conjecture
a Fano variety $X$ admits a Kähler-Einstein metric if and only if
$(X,-K_{X})$ is K-polystable. The cases when $X$ is non-singular
was established in \cite{c-d-s} and, very recently, the singular
case was settled in \cite{l-x-z}. As for the notion of Gibbs stability
it arose in the probabilistic construction of Kähler-Einstein metrics
on Fano manifolds proposed in \cite{berm8 comma 5}. Relations to
the Yau-Tian-Donaldson conjecture - in particular the relation between
Gibbs stability and K-stability - are discussed in the survey \cite{ber11}.
Here we will just highlight the role of holomorphic vector fields,
as a motivation for Theorem \ref{thm:gibbs intro}. First of all,
the assumption that $X$ is Gibbs stable at level $k$ is precisely
what ensures that $\mu^{(N_{k})}$ in formula \ref{eq:def of mu N and Z N}
is a well-defined probability measure on $X^{N_{k}}.$ Assuming that
this is the case for $k$ sufficiently large, i.e that $X$ is Gibbs
stable, it is conjectured in \cite{berm8 comma 5} that the measure
on $X$ defined as the push-forward of $\mu^{(N_{k})}$ to $X$ converges
weakly, as $k\rightarrow\infty,$ to the normalized volume form $dV_{KE}$
of a Kähler-Einstein metric on $X.$ Since $\mu^{(N_{k})}$ is biholomorphically
invariant so is $dV_{KE}.$ As a consequence, $X$ can not admit any
non-trivial holomorphic vector field $V^{1,0}.$ Indeed, otherwise
its flow $F_{\tau}$ would act non-trivially on the volume form $dV_{KE}$
(as follows, for example, from the argument in Remark \ref{rem:vanishing in ample smooth case}).
The idea of the first proof of Theorem \ref{thm:gibbs intro} is to
show that the flow in fact also acts non-trivially on the measures
$\mu^{(N_{k})},$ even though $\mu^{(N_{k})}$ is singular. Thus,
in Theorem \ref{thm:gibbs intro} one only needs to assume that $X$
is Gibbs stable at\emph{ some} level $k.$ Could this, a priori weaker
assumption, imply that $X$ is Gibbs stable? After all, it seems natural
to expect that the divisor $\mathcal{D}_{k}$ becomes less singular
at $k$ is increased. More precisely, the answer would be affirmative
if the log canonical threshold $\text{lct}(\mathcal{D}_{k})$ of $\mathcal{D}_{k}$
were increasing with respect to $k,$ since $X$ is Gibbs stable at
level $k$ iff $\text{lct}(\mathcal{D}_{k})>1$ (the definition of
the log canonical threshold of a divisor is recalled in Remark \ref{rem:lct}).
For example, this is the case when $X=\P^{1},$ or, more generally,
for any Fano orbifold curve \cite{berm12}. Unfortunately, the behavior
of $\text{lct}(\mathcal{D}_{k})$ in higher dimensions appears, however,
to be rather elusive. 

Part of the aforementioned conjecture in \cite{berm8 comma 5} says
that any Gibbs stable Fano manifold admits a unique Kähler-Einstein
metric (leaving the convergence issue aside). A slightly weaker version
of this conjecture was established in \cite{f-o}, where it was shown
that if\emph{ $X$ is uniformly Gibbs stable} (in the sense that there
exists $\epsilon>0$ such that $\text{lct}(\mathcal{D}_{k})\geq1+\epsilon$
for any sufficiently large $k),$ then $X$ is uniformly K-stable
(and thus admits a unique Kähler-Einstein metrics, by the solution
of the Yau-Tian-Donaldson conjecture). In particular, $X$ is K-stable,
which, in turn, directly implies that $X$ admit no holomorphic vector
fields (using that a holomorphic vector field induces a test configuration
for $X$). This argument does not, however, apply if $X$ is merely
assumed to be Gibbs stable (since the problem whether Gibbs stability
implies K-stability is open). But in Section \ref{sec:The-third-proof}
a ``quantized'' analog of this line of argument is pursued, showing
the Gibbs stability at a sufficiently high level $k$ implies that
$X$ admits no holomorphic vector fields.

\subsection{Acknowledgments}

Thanks to Rolf Andreasson and Bo Berndtsson for discussions and helpful
comments. This work was supported by grants from the Knut and Alice
Wallenberg foundation, the Göran Gustafsson foundation and the Swedish
Research Council.

\section{The first proof of Theorem \ref{thm:gibbs intro} and vector fields
preserving singular metrics}

\subsection{\label{subsec:Setup}Setup}

\subsubsection{Vector fields and their flows $F_{\tau}$}

Let $X$ be a complex manifold and $V^{1,0}$ a holomorphic vector
field on $X$ of type $(1,0),$ i.e. if $z\in\C^{n}$ are local holomorphic
coordinates on $X,$ then $V^{1,0}=\sum_{i=1}^{n}v_{i}(z)\partial/\partial z_{i}$
for holomorphic functions $v_{i}(z).$ We will denote by $F_{\tau}$
the corresponding locally defined flow. When $X$ is compact this
means that $F_{\tau}$ is the globally defined holomorphic map from
$\C\times X$ to $X$ uniquely determined by
\begin{equation}
\frac{\partial}{\partial\tau}_{|\tau=0}F_{\tau}(x)=V_{|x}^{1,0},\,\,\,\,F_{0}=I.\label{eq:derivative of flow}
\end{equation}
 In case $X$ is non-compact $F_{\tau}(x)$ is only defined in a neighborhood
of a given $(0,x)\in\C\times X.$ Given a holomorphic line bundle
$L\rightarrow X$ a vector field $V^{1,0}$ is said to \emph{lift
holomorphically to $L$} if there exists a holomorphic vector field
on the total space of $L$ whose local flow is $\C^{*}-$equivariant
and covers the local flow $F_{\tau}$ of $V^{1,0}$ on $X.$ Abusing
notation slightly we will use the same notation $F_{\tau}$ for the
lifted flow on the total space of $L.$ 

Given a complex vector field $V$ we will denote by $\delta_{V}$
the corresponding contraction operator (interior product) acting on
forms: if $\alpha$ is a $p-$form, then $\delta_{V}\alpha$ is the
$p-1-$form obtained by plugging in $V$ into the first argument of
$\alpha.$ 

\subsubsection{Metrics on line bundles }

We will use additive notation for metrics. This means that we identify
an Hermitian metric $\left\Vert \cdot\right\Vert $ on $L$ with a
collection of local functions $\phi_{U}$ associated to a given covering
of $X$ by open subsets $U$ and trivializing holomorphic sections
$e_{U}$ of $L\rightarrow U:$ 
\[
\phi_{U}:=-\log(\left\Vert e_{U}\right\Vert ^{2}),
\]
 which defines a function on $U.$ Of course, the functions $\phi_{U}$
on $U$ do not glue to define a global function on $X,$ but the two-form
$\partial\bar{\partial}\phi_{U}$ is globally well-defined and coincides
with the curvature form of $\left\Vert \cdot\right\Vert $. Accordingly,
as is customary, we will symbolically denote by $\phi$ a given Hermitian
metric on $L$ and by $i\partial\bar{\partial}\phi$ its curvature
form. More generally, we will allow the metric to be singular in the
sense that we only demand that $\phi_{U}\in L_{loc}^{1}.$ This ensures
that $i\partial\bar{\partial}\phi$ is a well-defined current on $X.$
In other words, fixing a smooth metric $\phi_{0}$ of $L$ we will
assume that that $\phi-\phi_{0}\in L^{1}(X)$ (but we do not assume
that $i\partial\bar{\partial}\phi$ is a positive current). Given
a (possibly singular) metric on $L$ and a global holomorphic section
$s$ of $L,$ i.e. $s\in H^{0}(X,L)$ we have
\[
\left\Vert s\right\Vert ^{2}=\left|s_{U}\right|^{2}e^{-\phi_{U}},
\]
 where $s_{U}$ is the local holomorphic function corresponding to
$s,$ $s_{U}:=s/e_{U}.$ Abusing notation slightly, the previous expression
shall simply be denoted by $\left|s\right|^{2}e^{-\phi}.$ 
\begin{example}
A global holomorphic section $s$ on $L$ induces a singular metric
$\phi$ on $L$ locally defined by $\phi_{U}:=\log(\left|s/e_{U}\right|^{2}).$
Accordingly, the singular metric on $L$ shall be denoted by $\log(|s|^{2}),$
abusing notation slightly. By the Poincaré-Lelong formula its curvature
current is the current of integration along the regular part of the
zero-locus of $\Delta_{s}$ of $s$ in $X$ (including multiplicities)
and thus defines a positive current. More generally, a meromorphic
section of $L$ also induces a singular metric on $L,$ whose curvature
is the current defined by $\Delta_{s}$ (including multiplicities,
defined to be negative for poles of $s$) which is positive iff $s$
is holomorphic.
\end{example}

If $F$ is a $\C^{*}-$equivariant holomorphic self-map of $L$ then
$F$ acts on metrics by pull-back, which in the present additive notation
means that
\[
(F^{*}\phi)_{U}(x):=-\log\left(\left\Vert (Fs_{x})\right\Vert ^{2}(Fx)\right)
\]
It should be stressed that, in general, $(F^{*}\phi)_{U}$ does not
coincide with $F^{*}(\phi_{U})$ i.e. the pull-back of the local function
$\phi_{U}.$ However, this is the case if the trivializing section
$s$ is\emph{ invariant under $F,$} i.e. if $(Fs_{x})(Fx)=s_{F(x)}.$

\subsubsection{Metrics on $-K_{X},$ the corresponding measures on $X$ and anti-canonical
klt divisors.}

Given a smooth metric $\phi$ on $-K_{X}$ we will use the symbolic
notation $e^{-\phi}$ for the corresponding measure on $X,$ defined
as follows. Given local holomorphic coordinates $z$ on $U\subset X$
denote by $e_{U}$ the corresponding trivialization of $-K_{X},$
i.e. $e_{U}=\partial/\partial z_{1}\wedge\cdots\wedge\partial/\partial z_{n}.$
The metric $\phi$ on $-K_{X}$ induces, as in the previous section,
a function $\phi_{U}$ on $U$ and the measure in question is defined
by
\[
e^{-\phi_{U}}i^{n^{2}}dz\wedge d\bar{z},\,\,\,\,\,dz:=dz_{1}\wedge\cdots\wedge dz_{n},
\]
on $U,$ which glues to define a global measure on $X.$ In the case
when $\phi$ is a singular metric on $-K_{X}$ the corresponding measure
$e^{-\phi}$ is a globally well-defined measure on $X$ as long as
$e^{-\phi_{U}}\in L_{loc}^{1}.$ In particular, if $s_{k}$ is a holomorphic
section of $-kK_{X}$ then one obtains a measure on $X,$ symbolically
denoted by $|s_{k}|^{-2/k},$ if $|s_{k}|^{-2/k}$ is in $L_{loc}^{1}.$
If this is the case then the corresponding anti-canonical divisor
$\Delta$ on $X$ (i.e. $k^{-1}$ times the zero-locus of $s_{k})$
is said to be\emph{ klt. }By the well-known algebraic characterization
of klt divisors this is, in fact, equivalent to the condition that
the measure $|s_{k}|^{-2/k}$ is in $L_{loc}^{p}$ for some $p>1$
\cite{ko0}. 
\begin{rem}
\label{rem:lct}In general, if $\Delta$ is an effective divisor on
a complex manifold $X,$ cut out by a holomorphic section $s$ of
a line bundle $L\rightarrow X,$ then its\emph{ log canonical threshold}
$\text{lct \ensuremath{(\Delta)}}$may be analytically defined as
the sup over all $t\in]0,\infty[$ such that, locally, $|s|^{-2t}\in L^{1}.$
The definition of $\text{lct \ensuremath{(\Delta)}}$ is extended
to $\Q-$divisor by imposing linearity and a $\Q-$divisor $\Delta$
is said to be\emph{ klt }if $\text{lct \ensuremath{(\Delta)}}>1$
\cite{ko0}. 
\end{rem}

\subsubsection{Lifts to $\pm K_{X}$ and tensor products }

Any holomorphic vector field on $X$ admits a canonical lift to $K_{X}$
using that the flow $F_{\tau}$ on $X$ naturally acts by pull-back
on forms on $X.$ By duality, it also lifts to $-K_{X}.$ If $L_{i}$
is a collection of $r$ holomorphic line bundles over $X$ and $F^{(i)}$
are $\C^{*}-$equivariant holomorphic maps from $L_{i}$ to $L_{i}$
(thus covering one and the same holomorphic self-map $F$ of $X)$
then one naturally obtains a $\C^{*}-$equivariant holomorphic $\C^{*}-$equivariant
self-map on the line bundle over $X$ defined as the tensor product
$L_{1}+\cdots+L_{r},$ by locally demanding that $F$ preserves tensor
products. 

\subsection{\label{subsec:Gibbs-stability-and}Gibbs stability and invariant
anti-canonical divisors}

Let $X$ be a compact complex manifold and set 
\[
N_{k}:=\dim_{\C}H^{0}(X,-kK_{X}).
\]
Let $\det S^{(k)}$ be the holomorphic section of $-kK_{X^{N_{k}}}\rightarrow X^{N_{k}}$
defined by 
\begin{equation}
(\det S^{(k)})(x_{1},x_{2},...,x_{N_{k}}):=\det(s_{i}^{(k)}(x_{j})),\label{eq:def of Slater text}
\end{equation}
 in terms of a given basis $s_{1}^{(k)},...,s_{N_{k}}^{(k)}$ in $H^{0}(X,-kK_{X}).$
We will denote by $\mathcal{D}_{k}$ the anti-canonical divisor $\Q-$divisor
on $X^{N_{k}}$ cut out by the $k$ th root of $\det S^{(k)},$ i.e.
$\mathcal{D}_{k}$ is defined as $k^{-1}$ times the zero-locus of
$\det S^{(k)},$ including multiplicities. 
\begin{defn}
A complex manifold $X$ is said to be \emph{Gibbs stable at level
$k$ }if the anti-canonical divisor $\mathcal{D}_{k}$ on $X^{N_{k}}$
is klt. 
\end{defn}

We will adopt the analytic characterization of the klt condition,
recalled in Section \ref{subsec:Setup}. It amounts to the condition
that the measure on $X^{N_{k}}$ induced by $\det S^{(k)},$ symbolically
denoted by $\left|\det S^{(k)}\right|^{-2/k}$, is in $L_{loc}^{p}$
for some $p>1.$

It should be stressed that the divisor $\mathcal{D}_{k}$ is canonically
attached to $X,$ as made precise by the following 
\begin{lem}
\label{lem:The-zero-locus-of det section}The zero-locus of $(\det S^{(k)})$
in $X^{N_{k}}$ is independent of the choice of basis in $H^{0}(X,-kK_{X})$
and, as a consequence, it is preserved by any biholomorphism $F$
of $X$ acting diagonally on $X^{N_{k}}.$ 
\end{lem}

\begin{proof}
Changing the basis in $H^{0}(X,-kK_{X})$ has the effect of multiplying
the corresponding determinant section $(\det S^{(k)})$ of $-K_{X^{N_{k}}}\rightarrow X^{N_{k}},$
by a non-zero complex number, namely the determinant of the change
of basis matrix. Thus the zero-locus of $(\det S^{(k)})$ is not altered.
Next, if $F$ is a biholomorphism of $X$ we also denote by $F$ its
canonical lift to $-kK_{X}.$ It induces an action on $H^{0}(X,-kK_{X})$
and a biholomorphism of $X^{N_{k}},$ defined by 
\[
(F\cdot s)(x):=F^{-1}s(Fx),\,\,\,\,F(x_{1},...,x_{N_{k}})=(F(x_{1}),...,F(x_{N_{k}})),
\]
 respectively. Accordingly, $F\cdot(\det S^{(k)})$ coincides with
the determinant section defined with respect to the basis $F\cdot s_{1}^{(k)},...,F\cdot s_{N_{k}}^{(k)}$
of $H^{0}(X,-kK_{X}).$ By the previous step this shows that the zero-locus
of $\det S^{(k)}$ is preserved by the diagonal action on $X^{N_{k}}$
of $F.$ An alternative geometric proof of the lemma may be obtained
by noting that the following geometric description of the zero-locus
of $\det S^{(k)}(x_{1},...x_{N_{k}})$ holds: 
\[
\det S^{(k)}(x_{1},...x_{N_{k}})=0\iff\,\exists s_{k}\in H(X,-kK_{X}):\,\,s_{k}(x_{i})=0,\,\,s_{k}\not\equiv0.
\]
Indeed, the condition $\det S^{(k)}(x_{1},...x_{N_{k}})=0$ is equivalent
to the non-injectivity of the linear map from $H^{0}(X,-kK_{X})$
to $L_{x_{1}}\oplus...\oplus L_{x_{N_{k}}}$ defined by $s\mapsto(s(x_{1}),...,s(x_{N})).$
\end{proof}
In order to prove Theorem \ref{thm:gibbs intro} (i.e. that Gibbs
stability of $X$ with $-K_{X}$ big, at some level $k,$ implies
that $X$ admits no non-trivial holomorphic vector fields) it will
be enough to prove the following proposition:
\begin{prop}
\label{prop:invariant klt divisors}Let $X$ be a compact complex
manifold such that $-K_{X}$ is big and $V^{1,0}$ a non-trivial holomorphic
vector field on $X$ whose flow $F_{\tau}$ preserves an analytic
subvariety of $X$ of codimension one. Then the analytic subvariety
is not the support of an anti-canonical klt divisor $D.$ Equivalently,
if $F_{\tau}$ preserves an anti-canonical divisor $D,$ then $D$
is not klt.
\end{prop}

Indeed, if $V^{1,0}$ is a holomorphic vector field on $X,$ then
we ca apply the previous proposition to the induced diagonal flow
on $X^{N_{k}},$ which, by the previous lemma, preserves the divisor
$\mathcal{D}_{k}$ on $X^{N_{k}}.$ In turn, the previous proposition
follows from Theorem \ref{thm:mu intro}, stated in the introduction,
applied to the probability measure 
\[
\mu=\frac{|s_{k}|^{-2/k}}{\int_{X}|s_{k}|^{-2/k}},
\]
 where we have represented the $\Q-$divisor $D$ as the $k^{-1}$
times the zero-locus (including multiplicities) of a holomorphic section
$s_{k}$ of some power $k$ th tensor power $-kK_{X}$ and $\mu$
denotes the corresponding measure (see Section \ref{subsec:Setup}).
Indeed, since the zero-locus $s_{k}$ is invariant under $F_{\tau}$
we have $F_{\tau}\cdot s_{k}=\rho(\tau)s_{k}$ for some non-vanishing
function $\rho(\tau)$ on $\C$ and hence $F_{k}^{*}\mu=\mu.$ The
proof of Theorem \ref{thm:mu intro} is given in the following section. 

\subsection{A vanishing result for holomorphic vector fields preserving singular
metrics}

Recall that a holomorphic line bundle $L$ over a compact complex
manifold $X$ is said to be \emph{big} if the space $H^{0}(X,kL)$
has maximal asymptotic growth:
\[
\dim H^{0}(X,kL)\geq Ck^{\dim X}
\]
 for some positive constant $C$ (the converse inequality always holds).
In analytic terms this equivalently means that $L$ admits a (possibly
singular) metric whose curvature current is strictly positive (i.e.
it is bounded from below by a smooth Hermitian metric on $X).$ The
following result contains, in particular, Theorem \ref{thm:mu intro}
(by letting $\phi$ be the metric on $-K_{X}$ induced by the measure
$\mu).$ 
\begin{thm}
\label{thm:sing metric Fano}Let $X$ be a compact complex manifold
$X$ and $V^{1,0}$ a holomorphic vector field on $X.$ Assume that
the flow of $V^{1,0}$ preserves some (possibly singular) metric $\phi$
on $-K_{X},$ i.e. $F_{\tau}^{*}\phi=\phi.$ If $e^{-p\phi_{U}}$
is locally integrable for some $p>1,$ then the flow of $V^{1,0}$
acts trivially on the complex vector space $H^{0}(X,-kK_{X})$ for
any positive integer $k.$ In particular, if $-K_{X}$ is big, then
$V^{1,0}$ is trivial, i.e. vanishes identically on $X.$ 
\end{thm}

\begin{proof}
Given $\epsilon>0$ denote by $\mathcal{N}_{\epsilon}$ the function
on $H^{0}(X,-kK_{X})$ defined by 
\[
\mathcal{N}_{\epsilon}(s_{k}):=\left(\int_{X}\left(\left|s_{k}\right|^{2}e^{-k\phi}\right)^{\epsilon/k}e^{-\phi}\right)^{k/2\epsilon},
\]
 which, by assumption, defines a finite function on $H^{0}(X,-kK_{X})$
when $\epsilon$ is sufficiently small (recall that $e^{-\phi}$ denotes
the measure on $X$ naturally attached to the metric $\phi$ on $-K_{X}).$
By the invariance assumption $\mathcal{N}_{k}$ is invariant under
the flow $F_{\tau}$ of $V^{1,0},$ i.e. $F_{\tau}^{*}\mathcal{N}_{\epsilon}=\mathcal{N}_{\epsilon}.$
Thus $\Omega:=\{\mathcal{N}_{\epsilon}<1$\} is a domain in the finite
dimensional complex vector space $H^{0}(X,-kK_{X}),$ which is invariant
under the flow $F_{\tau}$ of $V^{1,0}$ and bounded (since $\mathcal{N}_{\epsilon}$
is positively one-homogeneous of degree one). But then $F_{\tau}$
must be the identity map for any $\tau.$ Indeed, for any given $w\in\Omega$
the orbit $\tau\rightarrow$ $F_{\tau}(w)$ defines a holomorphic
map from $\C$ to $\C^{N_{k}}$ which is bounded (since $\Omega$
is bounded) and thus constant, i.e. equal to $F_{0}(w)(=w).$ 

Next, assume that $-K_{X}$ is big and fix a bases $s_{1}^{(k)}(x):\cdots:s_{N_{k}}^{(k)}(x)$
in $H^{0}(X,-kK_{X})$ and consider the corresponding Kodaira map:
\[
\Phi_{k}:X\dashrightarrow\P^{N_{k}-1},\,\,\,\,x\mapsto[s_{1}^{(k)}(x):\cdots:s_{N_{k}}^{(k)}(x)],
\]
 which is a holomorphic map on the complement in $X$ of the joint
zero-locus of the bases elements $s_{i}^{(k)}.$ If $F$ is a biholomorphism
of $-K_{X},$ commuting with the $\C^{*}-$action on $-K_{X},$ then
it induces an invertible endomorphism of $H^{0}(X,-kK_{X})$ which
descends to a biholomorphism $[F]$ of $\P^{N_{k}-1}$ which intertwines
the Kodaira map $\Phi_{k},$ i.e. 
\begin{equation}
[F]\circ\Phi_{k}=\Phi_{k}\circ F.\label{eq:intertw}
\end{equation}
 In particular, if $F$ is taken as the flow $F_{\tau}$ of a holomorphic
vector field $V^{1,0}$ satisfying the assumptions in the theorem
then, by the first part of the theorem $[F_{\tau}]$ is the identity.
But then it follows from the assumption that $-K_{X}$ is big that
$F_{\tau}$ is the identity (i.e. $V^{1,0}$ is trivial). Indeed,
by Siegel's lemma \cite[Lemma 2.2.6]{m-m}), for $k$ sufficiently
large, $\Phi_{k}$ has maximal rank, which (by the implicit function
theorem) means that there exists an open subset $U$ of $X$ such
that $\Phi_{k}$ is a biholomorphism between $U$ and $\Phi_{k}(U)\subset Y.$
Since $[F_{\tau}]$ is the identity the intertwining property \ref{eq:intertw}
thus implies that $F_{\tau}$ is the identity on $U$ and hence everywhere
on $X,$ as desired.
\end{proof}

\begin{rem}
\label{rem:general type}The proof of the previous theorem can be
viewed as a variant of the proof in \cite{kob} that any complex manifold
of general type (i.e. such that $K_{X}$ is big) admits no non-trivial
holomorphic vector fields. Indeed, replacing $-K_{X}$ with $K_{X}$
and $e^{-\phi}$ with $e^{\phi}$ (which defines a measure on $X$
if $\phi$ is a metric on $K_{X})$ the corresponding function $\mathcal{N}_{\epsilon}$
is, for $\epsilon=1,$ independent of $\phi$ and thus preserved by
any holomorphic vector field. 
\end{rem}

In the case when the metric $\phi$ on $-K_{X}$ is locally of the
form 
\begin{equation}
\phi=\lambda_{i}\sum_{i=1}^{m}\log|f_{i}|+\psi\label{eq:psh plus log hol}
\end{equation}
 with $f_{i}$ holomorphic, $\lambda_{i}\in\R$ and $\psi$ plurisubharmonic
(i.e. $i\partial\bar{\partial}\psi\geq0)$ it is enough to assume
$\phi$ is in $L_{loc}^{1}$ in the previous theorem. Indeed, by the
resolution of the strong openness conjecture \cite{g-z} this implies
the $L_{loc}^{p}-$condition for some $p>1.$ To see the necessity
of the integrability assumption $e^{-\phi_{U}}\in L_{loc}^{1}$ when
$i\partial\bar{\partial}\phi\geq0$ assume that $X$ is toric (i.e.
$X$ is an equivariant compactification of the complex torus $\C^{*n}$).
Any toric variety is of Fano type and thus $-K_{X}$ is big. Let $\phi$
be the singular metric on $-K_{X}$ induced by the standard invariant
section $s$ in $H^{0}(X,-K_{X}),$ defined as the exterior product
of the generators of the action of $\C^{*n}$ on $X.$ This means
that on the dense subset $\C^{*n}\subset X$ 
\begin{equation}
s=z_{1}\frac{\partial}{\partial z_{1}}\wedge\cdots\wedge z_{n}\frac{\partial}{\partial z_{n}},\,\,\,e^{-\phi}=|s|^{-2}=|z_{1}|^{-2}\cdots|z_{n}|^{-2}\label{eq:s toric}
\end{equation}
The metric $\phi$ is preserved by the vector field $V^{1,0}$ on
$X$ induced by any element in the Lie algebra of $\C^{*n},$ but
$e^{-\phi_{U}}$ is not locally integrable. On the other hand, $e^{-\phi_{U}}$
is ``next to integrable'' in the sense that $e^{-\phi_{U}}$ is
in $L_{loc}^{p}$ for any $p<1.$ 

In the proof of Theorem \ref{thm:sing metric Fano} the fractional
power $\epsilon/k$ is needed unless $e^{-\phi}$ is in $L_{loc}^{p}$
for $p$ large and thus the proof corrects the argument in the proof
of \cite[Prop 6.5]{berm8 comma 5} where $\epsilon/k$ was taken to
be one. On the other hand the argument in \cite{berm8 comma 5} yields
the following result, which applies to any big line bundle $L:$
\begin{prop}
\label{prop:inv sing metric L}Let $L$ be a holomorphic line bundle
over a compact complex manifold $X$ and $V^{1,0}$ a holomorphic
vector field on $X$ that admits a holomorphic lift to $L.$ Assume
that the flow of $V^{1,0}$ preserves a (possibly singular) metric
$\phi$ on $L,$ i.e. $F_{\tau}^{*}\phi=\phi.$ If $e^{-k\phi_{U}}$
is locally integrable, then the flow of $V^{1,0}$ acts trivially
on the complex vector space $H^{0}(X,kL+K_{X}).$ In particular, if
$L$ is big, then there exists some integer $k$ with the following
property: if $e^{-k\phi_{U}}$ is integrable, then $V^{1,0}$ is trivial,
i.e. vanishes identically on $X.$ More precisely, $k$ can be taken
as the minimal integer with the property that the corresponding meromorphic
Kodaira map from $X$ to $H^{0}(X,kL+K_{X})$ has maximal rank (i.e.
equals the dimension of $X).$
\end{prop}

\begin{proof}
Let $\mathcal{N}_{k}$ be the positively one-homogeneous function
on $H^{0}(X,kL+K_{X})$ defined by 
\[
\mathcal{N}_{k}(s_{k}):=\left(\int_{X}\left|s_{k}\right|^{2}e^{-k\phi}\right)^{1/2},
\]
 which, by the integrability assumption, defines a finite function
on $H^{0}(X,kL+K_{X}).$ Here $\left|s_{k}\right|^{2}e^{-kl\phi}$
denotes the measure on $X$ naturally attached to the section $s_{k}$
of $kL+K_{X}$ and the metric $\phi$ on $L$ using that $s_{k}$
may be identified with a holomorphic top form with values in $kL.$
By the invariance assumption $\mathcal{N}_{k}$ is invariant under
the flow $F_{\tau}$ of $V^{1,0},$ i.e. $F_{\tau}^{*}\mathcal{N}_{k}=\mathcal{N}_{k}.$
One can then conclude precisely as before. Alternatively, in this
case, exploiting that $\mathcal{N}_{k}$ is an $L^{2}-$norm, one
can also conclude using the following trivial fact, applied to the
generator of the flow on $H^{0}(X,kL+K_{X})$ of the real part of
$V^{1,0}:$ if a matrix is both Hermitian and anti-Hermitian, then
it vanishes identically.

Next, assume that $L$ is big. Then there exists a positive constant
$c$ such that 
\begin{equation}
\dim H^{0}(X,kL+K_{X})\geq ck^{\dim X}\label{eq:lower bd on dim}
\end{equation}
Indeed, $L$ is big iff it admits a metric $\phi$ with analytic singularities
whose curvature defines a Kähler current \cite[Thm 2.3.30]{m-m} (i.e.
it is strictly positive). The lower bound \ref{eq:lower bd on dim}
then follows from Bonavero's singular holomorphic Morse inequalities
\cite[Cor 2.3.26]{m-m} (which apply to $kL+E$ for any given holomorphic
line bundle $E$). It then follows, exactly, as in the proof of the
previous theorem  that there exists $k_{0}$ such that the corresponding
Kodaira map $\Phi_{k_{0}}$ has maximal rank. This implies, just as
before, that the flow $F_{\tau}$ is trivial.
\end{proof}
The previous proposition applies, in particular, when $\phi$ is psh
and has vanishing Lelong numbers, which is equivalent to $e^{-\phi}\in L_{loc}^{p}$
for\emph{ any} $p>1.$ It should be stressed that for a general big
line bundle $L$ the weaker integrability condition in Theorem \ref{thm:sing metric Fano}
is, however, not sufficient. This is illustrated by the case when
$X$ is the simplest toric Fano manifold, $X=\P^{1},$ and $L=\mathcal{O}(1),$
identified with $-K_{X}/2,$ endowed with the metric $\phi$ induced
by the $\C^{*}-$invariant section $s$ in $-K_{X}$ (formula \ref{eq:s toric}).
Indeed, then $e^{-\phi}\in L_{loc}^{p}$ for any $p<2$ and yet the
metric $\phi$ on $L$ is invariant under the generator $V^{1,0}$
of the $\C^{*}-$ action on $X$ (since $s$ is). 

\section{\label{sec:The-second-proof}The second proof and generalized Hamiltonians}

As explained in Section \ref{subsec:Gibbs-stability-and} Theorem
\ref{thm:gibbs intro} follows from Theorem \ref{thm:sing metric Fano}.
We next show that in the case that $-K_{X}$ is big and nef on a projective
manifold $X$ and $\phi$ is assumed psh the latter theorem can be
deduced from the following result \cite[Prop 8.2]{bern}:
\begin{prop}
\label{prop:(Berndtsson).-Let-}(Berndtsson)\cite{bern}. Let $L$
be a holomorphic line bundle over a compact Kähler manifold $X$ such
that $H^{0}(X,L+K_{X})$ is non-trivial and $H^{1}(X,L+K_{X})$ is
trivial. If $V^{1,0}$ is a holomorphic vector field on $X$ such
that

\[
\delta_{V^{1,0}}\partial\bar{\partial}\phi=0,
\]
 for a (possibly singular) metric $\phi$ on $L$ with positive curvature
current $\partial\bar{\partial}\phi$ (i.e. $\phi$ is locally psh)
such that $e^{-\phi}\in L_{loc}^{1}$, then $V^{1,0}$ vanishes identically. 
\end{prop}

The proof of the previous proposition is based on delicate integration
by parts. The relation to Theorem \ref{thm:sing metric Fano} stems
from the following completely local result:
\begin{lem}
\label{lem:V preserving phi contracts to zero}Let $X$ be a (possibly
non-compact) complex manifold endowed with a holomorphic line bundle
$L$ and $\phi$ a (possibly singular) metric on $L.$ If $\phi$
is preserved by a holomorphic $(1,0)-$vector field $V^{1,0},$ then
\[
\delta_{V^{1,0}}\partial\bar{\partial}\phi=0
\]
in the sense of currents.
\end{lem}

\begin{proof}
Consider the global complex-valued generalized function (i.e. current
of degree zero) on $X$ defined by
\begin{equation}
h=\frac{\partial}{\partial\tau}_{|\tau=0}(F_{\tau})^{*}\phi,\label{eq:def of h in pf}
\end{equation}
(this is indeed a globally defined generalized function since the
difference between any two metrics is a globally well-defined function
in $L^{1}(X)$). Then we can express

\begin{equation}
h=\delta_{V^{1,0}}\partial\phi+f\label{eq:h as diff plus f}
\end{equation}
 where $\bar{\partial}f=0.$ Indeed, let $g(z,\tau)$ be the local
non-vanishing holomorphic function on $X\times\C$ defined by
\[
(F_{\tau}s_{z})(F_{\tau}z)=gs_{(F_{\tau}z)}.
\]
 Then 
\[
\left\Vert (F_{\tau}s_{z})(F_{\tau}z)\right\Vert =|g(z,\tau)|\left\Vert s_{(F_{\tau}z)}\right\Vert 
\]
 and hence
\[
((F_{\tau})^{*}\phi)_{U}(z)=((F_{\tau})^{*}\phi_{U})(z)-\log g(z,\tau)-\overline{\log g(z,\tau)},
\]
 giving 
\[
\frac{\partial}{\partial\tau}((F_{\tau})^{*}\phi)_{U}(z)=\frac{\partial}{\partial\tau}(F_{\tau})^{*}\phi_{U}-\frac{\partial}{\partial\tau}\log g(z,\tau).
\]
 Evaluating this at $\tau=0$ proves formula \ref{eq:h as diff plus f}
with $f(z):=-\frac{\partial}{\partial\tau}\log g(z,\tau)$ evaluated
at $\tau=0.$ Hence, applying $\bar{\partial}$ gives 
\begin{equation}
\bar{\partial}h=\delta_{V^{1,0}}\bar{\partial}\partial\phi+0,\label{eq:dbar h in pf}
\end{equation}
(using that $\bar{\partial}$ commutes with $\delta_{V^{1,0}}$ for
a holomorphic vector field), as desired.
\end{proof}
\begin{rem}
\label{rem:vanishing in ample smooth case}Of course, the assumption
in the previous lemma do not, in general, imply that $V^{1,0}\equiv0$
(as illustrated by the case when $(L,\phi)$ is trivial). However,
if one adds the global assumption that $L$ be \emph{ample} (positive)
line bundle over a compact manifold and $\phi$ is\emph{ smooth} then
there exist some point $x_{0}$ where $\partial\bar{\partial}\phi_{|x_{0}}>0$
(by the maximum principle) and thus $i\partial\bar{\partial}\phi>0$
on a whole neighborhood $U.$ The previous proposition then forces
$V^{1,0}\equiv0$ on $U$ and thus on all of $X,$ since $V^{1,0}$
is holomorphic. However, the point of Theorem \ref{thm:sing metric Fano}
and Prop \ref{prop:inv sing metric L} is that they apply to singular
situations.
\end{rem}

If $L=-K_{X}$ the space $H^{0}(X,L+K_{X})$ is automatically non-trivial.
Moreover, $H^{1}(X,L+K_{X})$ is, in general, trivial if $L$ is big
and nef, by the Kawamata--Viehweg vanishing theorem \cite[Thm 9.1.18]{la}.
Hence, applying the previous proposition and lemma to $L=-K_{X}$
yields an alternative proof of Theorem \ref{thm:sing metric Fano}
in the case when $-K_{X}$ is big and nef on a projective manifold
$X$ and $\phi$ is psh. Note, however, that when $-K_{X}$ is merely
assumed to be big the space $H^{1}(X,-K_{X}+K_{X})$ can be non-trivial.
In other words, it can happen that $h^{0,1}(X)\neq0,$ as illustrated
by the following example.
\begin{example}
\label{exa:X surface}For a projective complex surface $X$ with $-K_{X}$
big it follows from the Castelnuovo's rationality criterion that $h^{0,1}(X)=0$
iff $X$ is rational, which, in general need not be the case. For
example, when $X$ is a ruled surface (i.e. a $\P^{1}-$bundle) over
a curve $C$ of strictly positive genus $g$ the anti-canonical line
bundle $-K_{X}$ is big if $-C\cdot C>2g-2$ \cite{s-d,o-o}. But
$h^{0,1}(X)=h^{0,1}(C)=g>0,$ since $X$ is birationally equivalent
to $\P^{1}\times C$ and $h^{0,1}(X)$ is a birational invariant. 
\end{example}

\subsection{An extension of Berndtsson's vanishing result and generalized Hamiltonians}

We next observe that the converse to Lemma \ref{lem:V preserving phi contracts to zero}
holds if $X$ is assumed compact:
\begin{lem}
\label{lem:V contracting zero preserves a phi}Let $X$ be a compact
complex manifold endowed with a holomorphic line bundle $L$ and $\phi$
a (possibly singular) metric on $L.$ If 
\[
\delta_{V^{1,0}}\partial\bar{\partial}\phi=0,
\]
then $V^{1,0}$admits a holomorphic lift to $L$ preserving $\phi.$ 
\end{lem}

\begin{proof}
We first show that $V^{1,0}$ admits a $\C^{*}-$equivariant holomorphic
lift to $L.$ To this end let $\phi_{0}$ be a fixed smooth metric
on $L.$ Then there exists a smooth function $h_{0}$ such that 
\begin{equation}
\delta_{V^{1,0}}dd^{c}\phi_{0}=\bar{\partial}h_{0}.\label{eq:h not}
\end{equation}
 Indeed, defining $u:=\phi_{0}-\phi\in L^{1}(X)$ gives 
\[
\delta_{V^{1,0}}dd^{c}\phi_{0}=0+\delta_{V^{1,0}}\bar{\partial}\partial u=\bar{\partial}\left(\delta_{V^{1,0}}\partial u\right),
\]
 using that $\bar{\partial}V^{1,0}=0.$ Thus $\delta_{V^{1,0}}\partial\bar{\partial}\phi_{0}$
is a smooth $(0,1)-$form which is $\bar{\partial}-$exact in the
complex of currents on $X$ and hence, by standard regularization
results, also $\bar{\partial}-$exact with respect to smooth forms.
This proves the existence of a smooth function $h_{0}$ as in formula
\ref{eq:h not}. But, as is well-known this implies that $V^{1,0}$
admits a $\C^{*}-$equivariant holomorphic lift to $L$ (see \cite[Lemma 13]{ber0}).
Next, as explained in the proof of Lemma \ref{lem:V preserving phi contracts to zero},
the function $h$ associated to $(V^{1,0},\phi),$ defined by formula
\ref{eq:def of h in pf}, satisfies $\bar{\partial}h=\delta_{V^{1,0}}\bar{\partial}\partial\phi,$which,
by assumption, vanishes identically. Hence, $h$ is holomorphic and
since $X$ is compact it follows that $h$ is constant, $h=c.$ By
the very definition of $h$ this means that 
\[
\frac{\partial}{\partial\tau}_{|\tau=0}\left((F_{\tau})^{*}\phi-\text{\ensuremath{\Re(}}c\tau)\right)=0,
\]
 where $\text{\ensuremath{\Re(}\ensuremath{\cdot}})$ denotes the
real part. Hence, if we define a new holomorphic lift of $V^{1,0}$
to $L$ by adding a term proportional to the generator of the standard
$\C^{*}-$action on $L$ and denote its flow by $G_{\tau},$ then
$(G_{\tau})^{*}\phi=\phi,$ as desired.
\end{proof}
Combining this lemma with Theorem \ref{thm:sing metric Fano} and
Prop \ref{prop:inv sing metric L} thus yields the following partial
generalization of Prop \ref{prop:(Berndtsson).-Let-}, where, in particular,
$\phi$ is not assumed to be psh:
\begin{prop}
\label{prop:variant of Bernd}Let $L$ be a big holomorphic line bundle
over a compact complex manifold $X$ and $V^{1,0}$ a holomorphic
vector field on $X$ and assume that 

\[
\delta_{V^{1,0}}\partial\bar{\partial}\phi=0,
\]
 for a (possibly singular) metric $\phi$ on $L$ with positive curvature
current $\partial\bar{\partial}\phi$ (i.e. $\phi$ is locally psh).
Then $V^{1,0}$ vanishes identically if $L=-K_{X}$ and $e^{-\phi}\in L_{loc}^{p}$
for some $p>1.$ In general, $V^{1,0}$ vanishes identical if $e^{-\phi}\in L_{loc}^{k}$
for the minimal integer $k$ such that the meromorphic Kodaira map
from $X$ to $H^{0}(X,kL+K_{X})$ has maximal rank (i.e. equals the
dimension of $X).$
\end{prop}

This result generalizes Prop \ref{prop:(Berndtsson).-Let-} in the
case that $L=-K_{X}.$ However, in the case of a general line bundle
$L,$ assumed big and nef, the assumption that $H^{0}(X,kL+K_{X})$
be non-trivial is weaker than demanding that the rank of the Kodaira
map be maximal. This is illustrated by the example following Prop
\ref{prop:inv sing metric L} where $H^{0}(X,kL+K_{X})$ is non-trivial
for $k\geq2,$ while the maximal rank property holds for $k\geq3.$ 

\subsubsection{Generalized Hamiltonians}

Let $L\rightarrow X$ be a big holomorphic line bundle over a compact
complex manifold, $V^{1,0}$ a holomorphic vector field on $X$ that
lifts holomorphically to $L$ and denote by $F_{\tau}$ the $\C^{*}-$equivariant
flow on $L$ covering the flow on $X.$ For example, if $L=-K_{X}$
any holomorphic vector field $V^{1,0}$ admits a canonical lift (see
Section \ref{subsec:Setup}). Given a (possibly singular) metric $\phi$
on $L$ we can then associate a generalized (complex-valued) function
$h$ to $V^{1,0},$ defined by formula \ref{eq:def of h in pf}. Equation
\ref{eq:dbar h in pf} says that $h$ is a \emph{complex Hamiltonian}
for the vector field $V^{1,0},$ in a generalized sense. Moreover,
Theorem \ref{thm:sing metric Fano} and Prop \ref{prop:inv sing metric L}
show that the association $V^{1,0}\mapsto h$ is injective under the
corresponding integrability conditions on $e^{-\phi_{U}}.$

Next, denote by $V_{\text{}}$ the real vector field on $X$ defined
as 
\[
V:=\left(V^{1,0}-\overline{V^{1,0}}\right)
\]
 (i.e. the imaginary part of $V^{1,0}$). Then $h_{V^{1,0}}$ is real-valued
iff the flow of $V$ preserves the metric $\phi$ (and then $h$ may
be expressed as the derivative of $\phi$ along the flow of $JV/2$,
where $J$ denotes the complex structure on $TX).$ Moreover, equation
\ref{eq:dbar h in pf} then translates to
\[
dh=\delta_{V}i\partial\bar{\partial}\phi
\]
 showing that $h$ is a Hamiltonian for the vector field $V,$ in
a generalized sense. While Hamiltonians are usually associated to
vector fields $V$ that lift to $L$ and whose flow preserve a given
symplectic form, in the present setup $i\partial\bar{\partial}\phi$
is neither assumed to be smooth nor non-degenerate. Still, the results
above reveal that the association $V\mapsto h$ is injective under
the appropriate integrability assumption on $e^{-\phi_{U}}.$

\section{\label{sec:The-third-proof}The third proof and coercivity of the
quantized Ding functional}

In this section yet another proof of Theorem \ref{thm:gibbs intro}
is provided, under the condition that $k$ is sufficiently large to
ensure that the corresponding Kodaira map is a biholomorphism onto
its image. But we start with a more general setup where $X$ is assumed
to be a compact complex manifold and $k$ is any given positive integer
such that the dimension $N_{k}$ of $H^{0}(X,-kK_{X})$ is strictly
positive. To simplify the notation we will will abbreviate $N_{k}=N.$
We fix once and for all a smooth metric $\phi_{0}$ on $-K_{X},$
a bases $s_{1}^{(k)},...,s_{N}^{(k)}$ in $H^{0}(X,-kK_{X})$ and
an Hermitian metric $H_{0}$ on $H^{0}(X,-kK_{X})$ making $s_{1}^{(k)},...,s_{N}^{(k)}$
orthonormal. Then, for any positive number $\gamma$ we set 
\[
\mathcal{Z}_{N,-\gamma}=\int_{X^{N}}\left\Vert \det S^{(k)}\right\Vert _{\phi_{0}}^{-2\gamma/k}(e^{-\phi_{0}})^{\otimes N},
\]
 where $\left\Vert \det S^{(k)}\right\Vert _{\phi_{0}}$ denotes the
point-wise norm of the section $\det S^{(k)}$ of $-kK_{X^{N}}$ (formula
\ref{eq:slater determinant}) with respect to the metric on $-kK_{X^{N}}$
induced by the fixed metric $\phi_{0}$ on $-K_{X}.$ As before, $e^{-\phi_{0}}$
denotes the corresponding volume form on $X.$ The definition is made
so that $\mathcal{Z}_{N,-1}$ coincides with the normalization constant
$\mathcal{Z}_{N}$ appearing in formula \ref{eq:def of mu N and Z N}
(since the contributions from the norm and the volume form cancel
when $\gamma=1).$ Thus $X$ is Gibbs stable at level $k$ iff there
exists $\gamma>1$ (depending on $k)$ such $\mathcal{Z}_{N,\gamma}<\infty.$
We will apply the following result in \ref{eq:def of mu N and Z N},
expressed in terms of the \emph{twisted quantized Ding functional}
$D_{k,\gamma}$ on the space $\mathcal{H}_{k}$ of all Hermitian metrics
$H_{k}$ on the $N-$complex vector space $H^{0}(X,-kK_{X}):$ 
\begin{thm}
\label{thm:inequality Z}\cite[Thm 3.3]{ber13} The following inequality
holds:
\[
-\frac{1}{\gamma N}\log\mathcal{Z}_{N_{k}}(-\gamma)\leq\inf_{\mathcal{H}_{k}}D_{k,-\gamma}+\frac{1}{kN}\log N.
\]
\end{thm}

We recall that the functional $D_{k,-\gamma}$ on $\mathcal{H}_{k}$
is defined by 
\begin{equation}
D_{k,-\gamma}(H):=\frac{1}{kN}\log\det_{ij\leq N}\left(H(s_{i}^{(k)},s_{j}^{(k)})\right)-\frac{1}{\gamma}\log\int_{X}e^{-\left(\gamma FS(H_{k})+(1-\gamma)\phi_{0}\right)},\label{eq:def of D k gamma on H k}
\end{equation}
 where $FS(H)$ is the metric on $-K_{X}$ induced from the Fubini-Study
metric on $\mathcal{O}(1)\rightarrow\P^{N-1}$ corresponding to $H:$

\begin{equation}
FS(H):=k^{-1}\log\left(\frac{1}{N}\sum_{i=1}^{N}\left|s_{i}^{H}\right|^{2}\right)\label{eq:def of fs}
\end{equation}
where $(s_{i}^{H})$ is any basis in $H^{0}(X,-kK_{X})$ which is
orthonormal wrt $H$ (the metric $FS(H)$ is smooth on the complement
of the base locus of of $-kK_{X}).$ The normalization by $N$ used
here is non-standard.
\begin{rem}
The \emph{quantized Ding functional $D_{k,-1}$} was introduced in
\cite{bbgz} (using a different notation), building on \cite{d1}.
The case of a general $\gamma$ is studied in \cite{rtz}, where $D_{k,\gamma}$
corresponds to the functional denoted by $F_{m}^{f,\delta}$ with
$m=k$ and $\delta=\gamma.$ In these works it is usually assumed
that $L$ is ample and $kL$ is base point free, but all the properties
that we shall use extend to the case when $N_{k}>0$ (with the same
proofs). 
\end{rem}

The functional $D_{k,-\gamma}$ is invariant under scaling by positive
numbers: 
\begin{equation}
D_{k,-\gamma}(e^{c}H)=D_{k,-\gamma}(H)\,\,\,\forall c\in\R,\label{eq:scale invariance of D k}
\end{equation}
 i.e. $D_{k,-\gamma}$ descends to the quotient space $\mathcal{H}_{k}/\R.$ 

Now assume that $X$ is Gibbs stable at level $k.$ Then it follows
from the inequality in the previous theorem that there exists $\gamma>1$
such that $D_{k,-\gamma}$ is bounded from below on $\mathcal{H}_{k},$
i.e. there exists a constant $C$ such that 
\[
D_{k,-\gamma}\geq-C.
\]
 This implies that the quantized Ding functional $D_{k,-1}$ is \emph{coercive}
on $\mathcal{H}_{k}/\R$ in the following sense: there exists $\epsilon>0$
such that 
\begin{equation}
D_{k,-1}\geq\epsilon J_{k}-C\label{eq:coerc ineq}
\end{equation}
 (after perhaps increasing $C),$ where $J_{k}$ is the function on
the space $\mathcal{H}_{k}/\R$ defined by 
\[
J_{k}(H):=\frac{1}{kN}\log\det_{ij\leq N}\left(H(s_{i}^{(k)},s_{j}^{(k)})\right)+\sup_{X}(FS(H)-\phi_{0}).
\]
Indeed, this follows from Hölder's inequality, precisely as in the
proof of \cite[Prop 4.9]{rtz}. Thus Theorem \ref{thm:gibbs intro}
(in its weaker form discussed above) follows from the following:
\begin{prop}
Assume that the quantized Ding functional $D_{k,-1}$ is coercive
on\emph{ $\mathcal{H}_{k},$ }i.e. that inequality \ref{eq:coerc ineq}
holds. Then the flow of $V^{1,0}$ acts trivially on the complex vector
space $H^{0}(X,-kK_{X}).$ In particular, if $-K_{X}$ is big and
$k$ is sufficiently large to ensure that the corresponding Kodaira
map is a bimeromorphism onto its image, then $X$ admits no non-trivial
holomorphic vector fields.
\end{prop}

\begin{proof}
Given a holomorphic vector field $V^{1,0}$ denote by $F_{\tau}$
the corresponding $\C^{*}-$equivariant canonical flow on $-K_{X}\rightarrow X$
and by $[F_{\tau}]$ the corresponding one-parameter subgroup of the
automorphism group $GL\left(H^{0}(X,-kK_{X})\right))$ of the vector
space $H^{0}(X,-kK_{X}).$ We shall identify $[F_{\tau}]$ with a
one-parameter subgroup of $GL(N,\C)$ (depending on the fixed bases
in $H^{0}(X,-kK_{X})$).

\emph{Step 1: $D_{k,-1}(F_{\tau}^{*}H_{0})$ is independent of $\tau.$}

To see this first note that, in general, $\tau\mapsto D_{k,-1}(F_{\tau}^{*}H_{0})$
is harmonic on $\C.$ Indeed, for any $\tau$
\begin{equation}
\int_{X}e^{-FS(F_{\tau}^{*}H_{0})}=\int_{X}F_{\tau}^{*}\left(e^{-FS(H_{0})}\right)=\int_{X}e^{-FS(H_{0})}\label{eq:integral invariant under pullback}
\end{equation}
and 
\[
\det_{ij\leq N}\left((F_{\tau}^{*}H_{0})(s_{i}^{(k)},s_{j}^{(k)})\right)=\left|\det_{ij\leq N}[F_{\tau}]\right|^{2}.
\]
 Thus 
\[
D_{k,-1}(F_{\tau}^{*}H_{0})-D_{k,-1}(H_{0})=\frac{1}{kN_{k}}\log\left(\left|\det_{ij\leq N_{k}}[F_{\tau}]\right|^{2}\right),
\]
 which is harmonic since $\tau\mapsto\det_{ij\leq N_{k}}[F_{\tau}]$
is a non-zero holomorphic function. But, in the present situation
$D_{k,-1}$ is bounded from below on $\mathcal{H}_{k}$ and thus the
harmonicity implies that $D_{k,-1}(F_{\tau}^{*}H_{0})$ is constant,
as desired. 

\emph{Step 2: If $\tau\mapsto J_{k}(F_{\tau}^{*}H_{0})$ is bounded
from above, then $[F_{\tau}]=e^{a\tau}I$ for some $a\in\C,$ where
$I$ denotes the identity in $GL(N,\C).$ }

First observe that $J_{k}$ defines an exhaustion function on $\mathcal{H}_{k}/\R.$
Indeed, for any volume form $\nu$ on $X$ there exists a constant
$c$ such 
\[
\sup_{X}u\geq\int_{X}u\nu-c\,\,\forall u\in PSH(X,\omega_{0}),
\]
where $\omega_{0}$ denotes the curvature form of $\phi_{0}$ and
$PSH(X,\omega_{0})$ denotes the space of all functions $u$ in $L^{1}(X)$
such that $i\partial\bar{\partial}u+\omega_{0}\geq0$ in the sense
of currents (this is standard and follows, for example, from the local
submean property of plurisubharmonic functions using a partition of
unity). Thus by the proof of \cite[Prop 3]{d1} (see also \cite[Lemma 7.6]{bbgz})
$J_{k}$ is an exhaustion function on $\mathcal{H}_{k}/\R.$ Next,
consider the fibration
\[
GL\left(N,\C\right)\rightarrow\mathcal{H}_{k},\,\,\,A\mapsto A^{*}H_{0}(=A^{*}A),
\]
whose fibers are isomorphic to $U(N)$ and thus compact. Hence, $J_{k}$
induces an exhaustion function of $GL\left(N,\C\right)/\C^{*}$ that
we denote by $j_{k}.$ Now, consider the function $\tau\mapsto\det[F_{\tau}].$
Since this is a one-parameter subgroup of $\C^{*}$ it can be expressed
as $e^{-a/N\tau}$ for some $a\in\C.$ Thus
\[
[\widetilde{F_{\tau}}]:=e^{a\tau}[F_{\tau}],
\]
 has unit-determinant, i.e. it defines a subgroup of $SL(N,\C).$
Since $D_{k,-1}$ is invariant under scaling the previous step implies
that $\tau\mapsto D_{k,-1}(\widetilde{F_{\tau}}^{*}H_{0})$ is constant.
Hence, by the assumed coercivity, $j_{k}([\widetilde{F_{\tau}}])$
is bounded from above. Since $j_{k}$ is an exhaustion function of
$GL(N,\C)/\C^{*}$ and $\widetilde{F_{\tau}}$ has unit determinant,
it follows that $\widetilde{F_{\tau}}$ takes values in a bounded
subset of $GL(N,\C).$ Hence, $\widetilde{F_{\tau}}$ is independent
of $\tau,$ i.e. $\widetilde{F_{\tau}}=I,$ which concludes the proof
of Step 2. Finally, the latter step implies, in fact, that $a=0,$
using that $F_{\tau}$ is the canonical lift of $V^{1,0}$ to $-K_{X}.$
Indeed, taking $\tau$ to be real and applying formula \ref{eq:integral invariant under pullback}
reveals that the real part of $a$ vanishes. Likewise, taking $\tau$
to be imaginary forces the imaginary part of $a$ to vanish. Thus
$V^{1,0}$ acts trivially on the complex vector space $H^{0}(X,-kK_{X}).$
The last statement then follows precisely as in the end of the proof
of Theorem \ref{thm:sing metric Fano}.
\end{proof}
When $-K_{X}$ is ample the application of the inequality in Theorem
\ref{thm:inequality Z} above may, alternatively, be replaced by the
following argument. First, by \cite[Thm 2.5]{f-o}, 
\[
\text{lct}(X^{N_{k}},\mathcal{D}_{k})\leq\delta_{k}(X),
\]
 where $\delta_{k}(X)$ an invariant introduced in \cite{f-o}. The
assumed Gibbs stability of $X$ at level $k$ means that $\text{lct}(X^{N_{k}},\mathcal{D}_{k})>1$
and thus it implies that $\delta_{k}(X)>1.$ In turn, this entails,
by \cite{rtz}, that the quantized Ding functional $D_{k,-1}$ is
coercive on $\mathcal{H}_{k}.$ Presumably, this argument also applies
when $-K_{X}$ is big. Anyhow, one virtue of the inequality in Theorem
\ref{thm:inequality Z} is that its proof is, essentially, elementary.

\section{\label{sec:Extension-to-the}Extension to the setting of log pairs}

\subsection{Setup}

We briefly recall the general setup of log pairs \cite{ko0,la,bbegz}.
By definition, a \emph{log pair }$(X,\Delta)$ is a normal variety
$X$ together with a $\Q-$divisor $\Delta$ such that $K_{X}+\Delta$
is $\Q-$Cartier, i.e. defines as a $\Q-$line bundle (called the\emph{
log canonical line bundle}). A log pair $(X,\Delta)$ is said to be\emph{
klt} if the following property holds for some (or equivalently any)
\emph{log resolution }$\pi$ of $(X,\Delta)$ i.e. a holomorphic bimeromorphism
from a non-singular variety $X'$ to $X$ such that $\pi^{*}\Delta+E$
has simple normal crossings, where $E$ denotes the sum of the exceptional
divisors of $\pi.$ Let $\Delta'$ be the divisor on $X'$ defined
by
\[
\pi^{*}(K_{X}+\Delta)=K_{X'}+\Delta'.
\]
By assumption $\Delta'$ has simple normal crossing and $(X,\Delta)$
is said to be\emph{ klt} if all the coefficients of $\Delta'$ are
strictly smaller than $1.$ A variety $X$ is said to have\emph{ log
terminal singularities} if the log pair $(X,0)$ is klt. 

Here we shall adopt the analytic characterization of the klt condition
of a log pair $(X,\Delta)$ which goes as follows (see \cite{bbegz}
for further background). First assume that $X$ is non-singular. The
$\Q-$divisor $\Delta$ induces a $\Q-$line bundle on $X,$ denoted
by the same symbol $\Delta,$ and a singular metric on the $\Q-$line
bundle $\Delta$, denoted by $\phi_{\Delta}$ (see Section \ref{subsec:Setup}).
Let now $\phi$ be a locally bounded metric on $-(K_{X}+\Delta).$
Then $\phi+\phi_{\Delta}$ defines a singular metric on $-K_{X}$
and thus also a measure on $X,$ denoted by $e^{-(\phi+\phi_{\Delta})}$.
The log pair $(X,\Delta)$ is klt iff the measure $e^{-(\phi+\phi_{\Delta})}$
on $X$ has finite mass. If $X$ is singular the measure $e^{-(\phi+\phi_{\Delta})}$
is first defined as before on the regular locus $X_{reg}$ of $X$
and then extended by zero to all of $X$ (so that it puts no mass
on the singular locus $X-X_{reg}).$ Then the previous discussion
applies. Note that the construction of the measure $e^{-(\phi+\phi_{\Delta})}$
is compatible with log resolutions, i.e. if $\pi$ is a log resolution
of $(X,\Delta),$ then 
\[
\pi_{*}\left(e^{-(\pi^{*}\phi+\phi_{\Delta'})}\right)=e^{-(\phi+\phi_{\Delta})}
\]
as measures. More generally, a divisor $D$ is said to be \emph{klt
wrt the log pair $(X,\Delta)$ }if $(X,D+\Delta)$ is a log pair which
is klt.

\subsubsection{Vector fields and log pairs}

By definition, a holomorphic vector field $V^{1,0}$ on a normal variety
$X$ is holomorphic vector field on the regular locus $X_{reg}$ of
$X$ such that there exists a holomorphic map $F_{\tau}$ from $\C\times X,$
called the \emph{flow} of $V^{1,0},$ satisfying formula \ref{eq:derivative of flow}
on $X_{reg}.$ If $(X,\Delta)$ is a log pair a holomorphic vector
field $V^{1,0}$ on $X$ is said to be\emph{ tangent to $\Delta$}
if it is tangent to $\Delta$ along the regular locus of $\Delta.$
This equivalently means that the flow $F_{\tau}$ of $V^{1,0}$ preserves
the support of $\Delta.$ Next assume that $\Delta$ has positive
integer coefficients. Equivalently, there exists a line bundle over
$X,$ also denoted by $\Delta,$ with a holomorphic section $s$ cutting
out $\Delta.$ If $F_{\tau}$ preserves the support of $\Delta$ then
there exists a canonical holomorphic $\C^{*}-$equivariant lift of
$F_{\tau}$ to the corresponding line bundle such that $s$ is invariant
under $F_{\tau}.$ Indeed, since $X-\Delta$ is preserved by $F_{\tau}$
the lift over $X-\Delta$ may be defined by 
\[
F_{\tau}(ws_{x}):=ws_{F_{\tau}(x)}
\]
 for any $w\in\C,$ using that $s$ defines a global trivialization
of the line bundle over $X-\Delta.$ This definition extends holomorphically
over the support of $\Delta$ (as can be seen by expressing $s$ in
terms of a fixed local trivialization of the line bundle over a neighborhood
of a given point in $\Delta).$ Finally, in the case that $\Delta$
has some negative integer coefficients the lifting of $F_{\tau}$
is achieved by first lifting $F_{\tau}$ to the dual of the corresponding
line bundle.

\subsection{Gibbs stability of log pairs}

Given a log pair $(X,\Delta)$ we fix a sufficiently divisible positive
integer $k,$ assuring that $k(K_{X}+\Delta)$ is a bona fide line
bundle. Then all the definitions in Section \ref{subsec:Gibbs-stability-and}
can be generalized by replacing the canonical line bundle $K_{X}$
of $X$ with the log canonical line bundle $K_{X}+\Delta$ of $(X,\Delta).$
For example, 
\[
N_{k}:=\dim H^{0}\left(X,-k(K_{X}+\Delta)\right)
\]
 and $(\det S^{(k)})$ denotes the holomorphic section of $-(K_{X^{N_{k}}}+\pi_{1}^{*}\Delta+...+\pi_{N_{k}}^{*}\Delta)$
defined as in formula \ref{eq:def of Slater text}, but now with $s_{1}^{(k)},...,s_{N_{k}}^{(k)}$
denoting a fixed basis in $H^{0}\left(X,-k(K_{X}+\Delta)\right).$
Accordingly, now $\mathcal{D}_{N_{k}}$ denotes the log anti-canonical
divisor on $X^{N_{k}}$ cut-out by $(\det S^{(k)})^{1/k}.$ 
\begin{defn}
A log pair $(X,\Delta)$ is said to be \emph{Gibbs stable at level
$k$ }if the anti-canonical divisor $\mathcal{D}_{k}$ on $X^{N_{k}}$
is klt with respect to the log pair $(X^{N_{k}},\pi_{1}^{*}\Delta+...+\pi_{N_{k}}^{*}\Delta).$ 
\end{defn}

The following result is a generalization of Theorem \ref{thm:gibbs intro}
to the setting of log pairs:
\begin{thm}
\label{thm:gibbs log }Assume that $-(K_{X}+\Delta)$ is big and that
$(X,\Delta)$ is Gibbs stable at some level $k.$ Then $X$ admits
no non-trivial holomorphic vector fields which are tangent to $\Delta.$ 
\end{thm}

\begin{rem}
In particular, taking $\Delta=0$ the previous theorem extends Theorem
\ref{thm:gibbs intro} to normal varieties $X.$ Note that the Gibbs
stability of $X$ directly implies that $(X,0)$ is klt, i.e. that
$X$ has log terminal singularities. This is analogous to the result
\cite[Thm 1.3]{od} saying that a K-semistable Fano variety has log
terminal singularities. 

To prove Theorem \ref{thm:gibbs log } denote by $\phi$ the singular
metric on $-(K_{X^{N_{k}}}+\pi_{1}^{*}\Delta+...+\pi_{N_{k}}^{*}\Delta)\rightarrow X^{N_{k}}$
corresponding to $\det S^{(k)}\in H^{0}\left(X,-k(K_{X}+\Delta)\right).$
By assumption, the diagonal flow $F_{\tau}$ on $X^{N_{k}}$ induced
by a given holomorphic vector field $V^{1,0}$ preserves the line
bundle $-k(K_{X^{N_{k}}}+\pi_{1}^{*}\Delta+...+\pi_{N_{k}}^{*}\Delta)$.
Hence, it follows precisely as in the proof of Lemma \ref{lem:The-zero-locus-of det section}
that the singular metric $\phi$ is invariant under $F_{\tau}.$ Theorem
\ref{thm:gibbs log } thus follows from the following generalization
of Theorem \ref{thm:sing metric Fano}, applied to $(X^{N_{k}},\pi_{1}^{*}\Delta+...+\pi_{N_{k}}^{*}\Delta).$
To simplify the statement of the integrability assumption we assume
that $\phi$ is, locally, of the form \ref{eq:psh plus log hol}:
\end{rem}

\begin{thm}
\label{thm:sing metric log Fano}Let $(X,\Delta)$ be a log pair and
$V^{1,0}$ a holomorphic vector field on $X,$ which is tangent to
$\Delta.$ Assume that the flow of $V^{1,0}$ preserves a (possibly
singular) metric $\phi$ on $-(K_{X}+\Delta),$ i.e. $F_{\tau}^{*}\phi=\phi.$
If $\phi$ is, locally, of the form \ref{eq:phi is psh plus log intro}
and the measure $e^{-(\phi+\phi_{\Delta})}$ on $X$ has finite total
mass, then the flow of $V^{1,0}$ acts trivially on the complex vector
space $H^{0}\left(X,-k(K_{X}+\Delta)\right)$ for any positive integer
$k.$ In particular, if $-(K_{X}+\Delta)$ is big, then $V^{1,0}$
is trivial, i.e. vanishes identically. 
\end{thm}

\begin{proof}
Given $\epsilon>0$ let $\mathcal{N}_{\epsilon}$ the positively one-homogeneous
function on $H^{0}\left(X,-k(K_{X}+\Delta)\right)$ defined by 
\[
\mathcal{N}_{\epsilon}(s_{k}):=\left(\int_{X}\left(\left|s_{k}\right|^{2}e^{-k\phi}\right)^{\epsilon/k}e^{-(\phi+\phi_{\Delta})}\right)^{k/2\epsilon},
\]
 defines a finite function on $H^{0}\left(X,-k(K_{X}+\Delta)\right)$
when $\epsilon$ is sufficiently small, by the assumption on $\phi$
and the resolution of the strong openness conjecture \cite{g-z} (applied
to a log resolution of $(X,\Delta)$). By the invariance assumption
$\mathcal{N}_{k}$ is invariant under the flow $F_{\tau}$ of $V^{1,0},$
i.e. $F_{\tau}^{*}\mathcal{N}_{\epsilon}=\mathcal{N}_{\epsilon}.$
Hence, we can proceed exactly as in the proof of Theorem \ref{thm:sing metric Fano}.
\end{proof}

\subsection{Examples: the case of log Fano curves}

It may be illuminating to consider Theorem \ref{thm:gibbs log } in
the simplest case where $X$ is a complex curve, i.e. $\dim X=1.$
Then 
\[
\Delta:=\sum_{1=1}^{m}x_{i}w_{i}
\]
 for given points $x_{1},...,x_{m}$ on $X$ and rational coefficients
$w_{i}.$ First consider the case when $X$ has genus zero, $X=\P^{1}.$
Then $-(K_{X}+\Delta)$ is ample iff 
\begin{equation}
2-\sum w_{i}>0.\label{eq:def of d L}
\end{equation}
 Moreover, as shown in \cite{berm12} $(X,\Delta)$ is Gibbs stable
at a sufficiently large level $k$ iff the following weight condition
holds: 
\begin{equation}
w_{i}<\sum_{i\neq j}w_{j},\,\,\,\forall i\label{eq:weight condition}
\end{equation}
(this condition first appeared in the existence result for conical
Kähler-Einstein metrics on $\P^{1}$ established in \cite{tr}). In
particular, given three points there always exists weights $w_{1},w_{2}$
and $w_{3}$ satisfying the inequalities \ref{eq:def of d L} and
\ref{eq:weight condition} and thus the corresponding log pair $(X,\Delta)$
is Gibbs stable at some level $k.$ Hence, Theorem \ref{thm:gibbs log }
implies the basic fact that any holomorphic vector field on $\P^{1}$
with three zeros vanishes identically (which follows, for example
from the fact that the space $H^{0}(X,-K_{X})$ of holomorphic vector
fields on $X$ is $3-$dimensional). 

Next, consider the case when $X$ has genus one (for genus strictly
larger than one there are no non-trivial holomorphic vector fields
at all). Given a point $x$ on $X$ let $\Delta$ be the divisor 
\[
\Delta=-k^{-1}x
\]
 for a fixed positive integer $k.$ Then $-(K_{X}+\Delta)$ has positive
degree (since $-K_{X}$ is trivial) and is thus ample. Moreover, $(X,\Delta)$
is Gibbs stable at level $k.$ Indeed, by construction, $-k(K_{X}+\Delta)=-k\Delta$
coincides with the line bundle of degree one over $X$ defined by
the point $x.$ It thus follows from the Riemann-Roch theorem that
$N_{k}=1.$ In other words, $-k(K_{X}+\Delta)$ admits a non-trivial
holomorphic section $s,$ uniquely determined up to multiplication
by $\C^{*}.$ Hence, 
\[
\left|\det S^{(k)}\right|^{-2/k}|s_{\Delta}|^{-2}=|s|^{-2/k}|s|^{+2/k}
\]
 which is trivially in $L^{1}(X),$ i.e. $(X,\Delta)$ is Gibbs stable
at level $k.$ Theorem \ref{thm:gibbs log } thus implies the basic
fact that there are no non-trivial holomorphic vector fields on $X$
preserving a given point $x$ on (which follows, for example, from
the fact that the space of holomorphic vector fields $H^{0}(X,-K_{X})$
on $X$ is one-dimensional and thus generated by the vector field
induced from the vector field $\partial/\partial z$ on $\C$ under
the standard isomorphism $X\simeq\C/\Lambda).$ 

\subsection{Log stability with respect to a parameter}

The Gibbs stability of a log pair $(X,\Delta)$ can, in fact, be defined
for some non-integer levels $k.$ Indeed, all that is needed is that
$-k(K_{X}+\Delta)$ defines a line bundle (i.e. a Cartier divisor).
For example, if 
\[
\Delta=(1-\gamma)S,\,\,\,\gamma\in]0,\infty[
\]
 for an anti-canonical divisor $S,$ then $-k(K_{X}+\Delta)$ is a
well-defined line bundle for any $k\in\gamma^{-1}\Z.$ Indeed, if
$k=\gamma^{-1}p$ for $p\in\Z,$ then 
\[
-k(K_{X}+\Delta)=-pK_{X}.
\]
In this case it is convenient to make the following
\begin{defn}
Let $S$ be anti-canonical divisor on a variety $X$ with log terminal
singularities and $\gamma\in]0,\infty[.$ Then $(X,S)$ is said to
be Gibbs stable at level $k$ with respect to the parameter $\gamma$
if $(X,(1-\gamma)S)$ is Gibbs stable at level $\gamma^{-1}k.$ Equivalently,
this means that the $\Q-$divisor $\gamma\mathcal{D}_{N_{k}}+(1-\gamma)\left(\pi_{1}^{*}S+\cdots\pi_{N_{k}}^{*}S\right)$
is klt on $X^{N_{k}},$ where $\mathcal{D}_{N_{k}}$ is the anti-canonical
divisor $\Q-$divisor on $X^{N_{k}}$ defined in Section \ref{subsec:Gibbs-stability-and}
and $\pi_{i}$ denotes the projection $X^{N_{k}}\rightarrow X$ from
$X^{N_{k}}$ onto the $i$th factor of $X^{N_{k}}.$
\end{defn}

One advantage of introducing the parameter $\gamma$ is that it can
be used as a deformation parameter (which from the point of view of
statistical mechanics plays the role of the negative of the inverse
temperature; see the discussion in \cite[Section 2.4]{berm12}). In
the following section we will provide a geometric application of log
Gibbs stability of $(X,S)$ with respect to a small parameter $\gamma.$ 
\begin{rem}
Formally, when $k=\infty$ the parameter $\gamma$ corresponds to
Donaldson's time-parameter in his singular generalization of Aubin's
method of continuity \cite{do2} for Kähler metrics $\omega_{\gamma}:$
\[
\mbox{\ensuremath{\mbox{Ric}}\ensuremath{\,\omega}}_{\gamma}=\gamma\omega_{\beta}+(1-\gamma)[S],
\]
 where $\mbox{\ensuremath{\mbox{Ric}}\ensuremath{\,\omega}}_{\gamma}$
denotes the Ricci curvature current of the singular Kähler metric
$\omega_{\gamma}$ and $[S]$ denotes the current of integration along
the divisor $S.$ As explained in \cite{berm12}, this formal analogy
can be made rigorous under a zero-free assumption on the zeta type
function appearing in formula \ref{eq:integr in pf theorem invariant non klt}
below (viewed as a meromorphic function of $\gamma).$
\end{rem}

\section{A vanishing theorem for holomorphic vector fields preserving anti-canonical
divisors that are not klt}

In this final section we apply the concept of log Gibbs stability
to prove Theorem \ref{thm:not klt intro}, stated in the introduction:
\begin{thm}
\label{thm:invariant non klt}Let $X$ be a compact complex manifold
such that $-K_{X}$ is big and $s$ a non-trivial holomorphic section
of $-K_{X},$ whose zero-locus $S$ is a non-singular connected subvariety
of $X.$ If $S$ is preserved by the flow a holomorphic vector field
$V^{1,0}$ on $X,$ then $V^{1,0}$ vanishes identically. More generally,
the result holds when the varieties $X$ and $S$ have log terminal
singularities and the divisor on $X$ defined by $s$ is irreducible.
\end{thm}

In contrast to Prop \ref{prop:inv sing metric L} the anti-canonical
divisor on $X$ defined by $s$ is not klt (but $(S,0)$ is klt, since
$S$ is assumed to have log terminal singularities). Without the regularity
and irreducibility assumption the result thus fails, in general. This
is illustrated by the toric example following Theorem \ref{thm:sing metric Fano},
where $s$ is the standard invariant anti-canonical divisor on $X,$
$S=X-\C^{*n}.$ In particular, when $X=\P^{1}$ the support $S$ of
this $\C^{*}-$invariant divisor consists of the two points $0$ and
$\infty$ in $\P^{1}$ and is thus non-singular, but not connected/irreducible.

\subsection{Proof of Theorem \ref{thm:invariant non klt}}

Fix a positive integer $k$ such that $N_{k}>1.$ Denote by $\pi_{i}$
the projection $X^{N_{k}}\rightarrow X$ from $X^{N_{k}}$ onto the
$i$th factor of $X^{N_{k}}.$ The section $s$ induces, by taking
tensor products, a holomorphic section $s^{\otimes N_{k}}$ of $-K_{X^{N_{k}}}.$
By Theorem \ref{thm:gibbs log } it is enough to show that $(X,(1-\gamma)S)$
is Gibbs stable at level $\gamma^{-1}k$ (i.e. that $(X,S)$ is Gibbs
stable at level $k$ with respect to the parameter $\gamma).$ In
other words, it is enough to show that
\begin{equation}
\int_{X^{N_{k}}}\left|\det S^{(k)}\right|^{-2\gamma/k}\left|s^{\otimes N_{k}}\right|^{-2(1-\gamma)}<\infty.\label{eq:integr in pf theorem invariant non klt}
\end{equation}
To this end we factorize
\[
\det S^{(k)}(x_{1},...,x_{N_{k}})=s(x_{1})^{\otimes l_{1}}s(x_{2})^{\otimes l_{2}}\cdots q(x_{1},x_{2},...,x_{N_{k}}),
\]
 where $q(x_{1},...x_{N_{k}})\in H^{0}\left(X^{N_{k}},(k-l_{1})\pi_{1}^{*}(-K_{X})+...(k-l_{N_{k}})\pi_{N_{k}}^{*}(-K_{X})\right)$
is not divisible by the irreducible divisor $s(x_{i})$ on $X^{N_{k}}$
for any $i.$ The integer $l_{1}$ coincides with the order of vanishing
of $\det S^{(k)}$ along $\{s(x_{1})=0\}\subset X^{N_{k}},$ for generic
$(x_{2},...,x_{N_{k}})$ and likewise for $i\neq1.$ Since $\det S^{(k)}(x_{1},...,x_{N_{k}})$
is totally antisymmetric it follows that $l_{1}=l_{i}$ for all $i.$
This means that there exists $l\leq k$ such that 
\[
\det S^{(k)}(x_{1},...,x_{N_{k}})=s(x_{1})^{\otimes l}s(x_{2})^{\otimes l}\cdots q(x_{1},x_{2},...,x_{N_{k}}),
\]
 where $q(x_{1},...x_{N_{k}})\in H^{0}(X^{N_{k}},-(k-l)K_{X^{N_{k}}}).$
The assumption $N_{k}>1$ forces $l<k.$ Indeed, otherwise, $q\in H^{0}(X^{N_{k}})$
and thus $q$ is constant, $q=C,$ giving $\det S^{(k)}(x_{1},...,x_{N_{k}})=Cs(x_{1})^{\otimes k}s(x_{2})^{\otimes k}\cdots.$
Since $\det S^{(k)}(x_{1},...,x_{N_{k}})$ is totally anti-symmetric
this can only happen if $N_{k}=1.$ Hence, 
\[
\left|\det S^{(k)}\right|^{-2\gamma/k}\left|s^{\otimes N_{k}}\right|^{-2(1-\gamma)}=\left|q\right|^{-2\gamma/k}\left|s(x_{1})\right|^{-2(1-\gamma+\gamma l/k)}\cdots\left|s(x_{N_{k}})\right|^{-2(1-\gamma+\gamma l/k)}
\]
By construction, for generic $x_{2},...,x_{N_{k}}$ the section $x\mapsto q(x,x_{2},...,x_{N_{k}})$
does not vanish identically along $\{s(x_{1})=0\}\subset X$ and likewise
for $i\neq1.$ Note that since $l<k$ the exponent $(1-\gamma+\gamma l/k)$
is strictly smaller than $1.$ As a consequence, we also have 
\[
k^{-1}=(k-l)^{-1}(k-l)/k<(k-l)^{-1}
\]
 and thus (since we may, without loss of generality, assume that $\left|q(x_{1},...,x_{N_{k}})\right|\leq1)$
\[
\left|\det S^{(k)}\right|^{-2\gamma/k}\left|s^{\otimes N_{k}}\right|^{-2(1-\gamma)}\leq\left(\left|q(x_{1},...,x_{N_{k}})\right|^{-2/(k-l)}\right)^{\gamma}\left|s(x_{1})\right|^{-2(1-\delta)}\cdots\left|s(x_{N_{k}})\right|^{-2(1-\delta)}
\]
 for some $\delta\in]0,1[.$ Next we will apply inversion of adjunction
\cite[Thm 7.5]{ko0} in the following form: let $B$ be a divisor
on $X$ such that the support of $S$ is not contained in the support
of $B,$ then 
\[
B_{|S}\,\text{klt close to \ensuremath{S\implies B+(1-\delta)S\,\text{klt \ensuremath{\text{close to \ensuremath{S}}}}}}
\]
for any $\delta\in[0,1[$ (in the case when $X$ and $S$ are non-singular
this is a direct consequence of the local Ohsawa-Takegoshi extension
theorem; see \cite[Section 7]{ko0}). Now set $B=\gamma Q,$ where
$Q$ is the anti-canonical divisor on $X$ defined as $(k-l)^{-1}$
times the divisor cut out by the holomorphic section $x\mapsto q(x,x_{2},...,x_{N_{k}})$
of $-(k-l)K_{X},$ for a fixed generic $(x_{2},...,x_{N_{k}})\in X^{N_{k}-1}.$
Then 
\[
\gamma<\alpha(S,-K_{X|S})\implies\gamma Q_{|S}\,\text{klt close to \ensuremath{S}},
\]
 where $-K_{X|S}$ denotes the restriction to $S$ of $-K_{X}$ and
$\text{lct}(Y,L)$ denotes the \emph{global log canonical threshold
}of a given line bundle $L$ over a variety $Y$ with log terminal
singularities (which, by \cite{dem}, coincides with Tian's \emph{alpha-invariant}
of $L\rightarrow Y):$ 
\[
\text{lct}(L):=\sup_{t>0}\left\{ t:\,\,tk^{-1}(q_{k}=0)\,\text{is \ensuremath{\text{klt \ensuremath{\forall}}q_{k}\in H^{0}(Y,kL),}\ensuremath{\forall k\in\Z_{+}}}\right\} ,
\]
 where $(q_{k}=0)$ denotes the divisor in $Y$ defined cut-out by
$q_{k},$ including multiplicities. On any open set $U$ in the complement
of the zero-locus of $S$ in $X$ we trivially have that 
\[
\gamma<\text{lct}(X,-K_{X})\implies\gamma Q+(1-\delta)S\,\text{klt \ensuremath{\text{on \ensuremath{U}}}.}
\]
Hence, 
\[
\gamma<\min\left\{ \text{lct}(X,-K_{X}),\text{lct}(S,-K_{X|S})\right\} \implies\int_{X}\left|\det S^{(k)}\right|^{-2\gamma/k}\left|s^{\otimes N_{k}}\right|^{-2(1-\gamma)}<\infty.
\]
 It then follows from the compactness of $X$ and the semi-continuity
of integrability thresholds (just as in the appendix of \cite{ber1})
that there exists a constant $C_{\gamma}$ only depending on $\gamma$
such for any fixed generic $(x_{2},...,x_{N_{k}})\in X^{N_{k}-1}$
\[
\int_{x\in X}\left|\det S^{(k)}(x,x_{2},...,x_{N_{k}})\right|^{-2\gamma/k}\left|s(x)\right|^{-2(1-\gamma)}\leq C_{\gamma}\left(\sup_{x\in X}\left\Vert \det S^{(k)}(x,x_{2},...,x_{N_{k}})\right\Vert \right)^{-2\gamma/k}
\]
 for a given fixed smooth metric $\left\Vert \cdot\right\Vert $ on
$-K_{X}.$ As a consequence, the inequality, in fact holds for\emph{
any} fixed $(x_{2},...,x_{N_{k}})\in X^{N_{k}-1}.$ Iterating this
inequality $N_{k}-1$ times (replacing the index $i=1$ with any index
$i$) thus gives 
\[
\int_{X^{N_{k}}}\left|\det S^{(k)}\right|^{-2\gamma/k}\left|s^{\otimes N_{k}}\right|^{-2(1-\gamma)}\leq C_{\gamma}^{N_{k}}\left(\sup_{X^{N_{k}}}\left\Vert \det S^{(k)}\right\Vert \right)^{-2\gamma/k}<\infty,
\]
 which concludes the proof of the theorem, using that $\text{lct}(Y,L)>0$
if $Y$ has log terminal singularities and $\dim H^{0}(Y,kL)\geq1$
for some positive integer $k.$ Indeed, if $q_{k}$ is a non-trivial
element in $H^{0}(Y,kL),$ then the log pair $\left(Y,tk^{-1}(q_{k}=0)\right)$
is klt when $t$ is a sufficiently small positive number, since, $(Y,0)$
is assumed klt and the klt condition is stable under small perturbations.

\end{document}